\documentclass[12pt]{amsart}
\usepackage{amsmath,amssymb, amsfonts,amscd,
psfrag,graphicx}

\newcommand{\lebn}

\theoremstyle{plain}

\theoremstyle{definition}

\numberwithin{equation}{section}

\usepackage{float}
\usepackage{url}
\usepackage{cite}

\usepackage{graphicx}%
\usepackage{multirow}%
\usepackage{amsmath,amssymb,amsfonts}%
\usepackage{amsthm}%
\usepackage{mathrsfs}%
\usepackage[title]{appendix}%
\usepackage{xcolor}%
\usepackage{textcomp}%
\usepackage{manyfoot}%
\usepackage{booktabs}%
\usepackage{algorithm}%
\usepackage{algorithmicx}%
\usepackage{algpseudocode}%
\usepackage{listings}%

\usepackage{multicol}

\usepackage{pifont}

\newcommand\raisepunct[1]{\,\mathpunct{\raisebox{0.5ex}{#1}}}
\usepackage{url}
\usepackage{booktabs}

\usepackage{xcolor}
\newtheorem{theorem}{Theorem}[section]

\newtheorem{definition}{Definition}[section]

\newtheorem{corollary}{Corollary}[section]
\numberwithin{equation}{section}

\usepackage{cite}
\usepackage{cleveref}
\crefformat{section}{\S#2#1#3} % see manual of cleveref, section 8.2.1
\crefformat{subsection}{\S#2#1#3}
\crefformat{subsubsection}{\S#2#1#3}
\usepackage{float}
\usepackage{CJKutf8}

\begin{document}

\bibliographystyle{plain}

\title[Generalized Tribonacci hyperbolic spinors]{Generalized Tribonacci hyperbolic spinors}
\author{Zehra \.{I}\c{S}B\.{I}L\.{I}R}
\address{Department of Mathematics, D\"uzce University,
D\"uzce, 81620, T\"urk{\.{\.i}}ye.}
\email{zehraisbilir@duzce.edu.tr}
\author{Bahar DO\u{G}AN YAZICI}
\address{Department of Mathematics, B{\.{\.i}}lec{\.{\.i}}k \c{S}eyh Edebal{\.{\.i}} University,
B{\.{\.i}}lec{\.{\.i}}k, 11000, T\"urk{\.{\.i}}ye.}
\email{bahar.dogan@bilecik.edu.tr}
\author{Murat TOSUN}
\address{Department of Mathematics, Sakarya University,
Sakarya, 54187, T\"urk{\.{\.i}}ye.}
\email{tosun@sakarya.edu.tr}


\date{\today}

\begin{abstract}
In this study, we introduce the generalized Tribonacci hyperbolic spinors and properties of this new special numbers system by the generalized Tribonacci numbers, which are one of the most general form of the third-order recurrence sequences, generalized Tribonacci quaternions, and hyperbolic spinors, which have quite an importance and framework from mathematics to physics. This study especially improves the relations between the hyperbolic spinors and generalized Tribonacci numbers with the help of the generalized Tribonacci split quaternions. Furthermore, we examine some special cases of them and construct both new equalities and fundamental properties such as recurrence relation, Binet formula, generating function, exponential generating function, Poisson generating function, summation formulas, special determinant properties, matrix formula, and special determinant equations. Also, we give some numerical algorithms with respect to the obtained materials. In addition to these, we give a brief introduction for further research: generalized Tribonacci  polynomial hyperbolic spinor sequence.
\end{abstract}

\maketitle
%%%%%%%%%%%%%%%%%%%%%%%%%%%%%%%%%%%%%%%%%%%%%%%%%%%%%%%%%%%%
\section{Introduction}
The theory of numbers has been the most important workframe in lots of disciplines, such as architecture, engineering, computer sciences, geometry, graph theory, etc., due to several beneficial applications. In the existing literature, many and varied studies have been completed and are ongoing with respect to number systems, especially special recurrence sequences. One of the most fundamental concept in the number system is quaternions, which were investigated by W. R. Hamilton to extend the complex numbers. The algebra of quaternions is an associative, non-commutative, and 4-dimensional Clifford algebra. The quaternion (real/Hamilton type) set is represented by $\mathbb{H}$ and defined as $\mathbb{H}=\{q\,\,|\,\,q = {q_0} + {q_1}i + {q_2}j + {q_3}k, \,\, q_0,q_1,q_2,q_3\in \mathbb{R}\}$,
where $i,j,k$ are real quaternionic units that satisfy the multiplication rules \cite{Hamilton,Hamilton2,Hamilton3}:
\begin{equation*}
i^2 = {j^2} = {k^2} =  - 1, \quad
ij =  - ji = k, \quad jk =  - kj =  i, \quad ki =  - ik =  j.
\end{equation*}
In addition to these, J. Cockle determined the split quaternions \cite{JCockle}. The split quaternionic units hold the rules \cite{JCockle,Diskaya2}:
\begin{equation}\label{splitunits}
\begin{array}{cc}
    i^2=-1, \quad  j^2=k^2=1, \quad ijk=1, \\  ij=-ji=k, \quad jk=-kj=-i, \quad  ki=-ik=j.
    \end{array}
\end{equation}

Then, generalized quaternions, which include both real and split quaternions and also have some special other quaternion types, are obtained \cite{Jafari3, Jafari2, Jafari1, Pottmann, Savin, Unger,Jafaritez}. The set of generalized quaternions is represented by $\mathbb{H}_{\alpha\beta}$ and generalized quaternionic units satisfy the following rules:
\begin{equation}\label{gq}
\begin{array}{cc}
i^2 =-\alpha, \quad {j^2} =-\beta, \quad {k^2} =-\alpha\beta,  \\
ij = -ji = k,  \quad jk = -kj = \beta i,  \quad ki = -ik = \alpha j.
\end{array}
\end{equation}
In the following Table \ref{ct-11}, the classifications of generalized quaternions can be examined:
\begin{table}[H]\caption{Classification of generalized quaternions}\label{ct-11}
\centering
\begin{tabular}{|l|l|}
  \hline
For & Type of generalized quaternions \\
  \hline
$\alpha=\beta=1$ & Real (Hamilton) quaternion \cite{Hamilton,Hamilton2,Hamilton3}\\
  \hline
$\alpha=1, \quad  \beta=-1$& Split quaternion \cite{JCockle}\\
  \hline
$\alpha=1, \quad \beta=0$&Semi-quaternion \cite{Rosenfeld,mor}    \\
  \hline
 $\alpha=1, \quad \beta=0$& Split semi-quaternion \cite{Rosenfeld}\\
 \hline
  $\alpha= \beta=0$& $1/4$-quaternion \cite{Rosenfeld,Hamilton3}\\
  \hline
    \end{tabular}
\end{table}
On the other hand, special recurrence sequences with different orders are one of the most attractive and fundamental topics in number theory. From the beginning to the end of this article, we are interested in the third-order special recurrence sequences, that is, the generalized Tribonacci number (or sequence) (for short, GTN), which is the most general form of the third-order recurrence sequences and has several special cases. Additionally, the generalized Tribonacci sequence $\{V_n(V_0,V_1,V_2;r,s,t)\}_{n\ge 0}$ satisfied the following recurrence relation:
\begin{equation}\label{recurrencerelationgeneralizedgeneralized Tribonacci}
V_n=rV_{n-1}+sV_{n-2}+tV_{n-3},\quad n \ge 3
\end{equation}
with the initial conditions $V_0=a,V_1=b,V_2=c$ are arbitrary integers and $r,s,t$ are real numbers \cite{Soykanrst,cerda2}. 
In addition to these, some special subgroups and special cases of generalized Tribonacci numbers can be seen in Table~\ref{somespecialcasesofgtntablenew} and Table~\ref{somespecialcasesofgtn}, respectively \cite{cerda2, Moralesmatrixrepresentation, Morales2, Morales3, Shannon2, Shannon, Soykanthirdpell, Soykangeneralizedgrahaml, Soykanpellpadovan, Soykangeneralized3-primes, Soykanonfourspecialcases, Soykannarayana, Soykangeneralizedreverse, Soykannewsumformulas, Soykangeneralizedjp, Soykanrst, Soykanexplicit, Soykansumformulas, onlineansiklopedi,cerecedadet,moralescomplex,Cereceda,ShannonandHoradam,ShannonHoradamandAnderson,Shannoniterative,ShannonandWong,Soykanbinomgeneralized3primes,Soykanbinomialtransformofgeneralizedtribonacci,Soykanrecurrence,Soykanbinomialtransformgeneralizedreverse3primes,Soykansum,soykan2, shannon} (see Section \ref{basicconcepts}). 

On the other hand, spinors were determined by Ehrenfest in the 1920s in quantum physics \cite{Vaz}. Then, Cartan examined the spinors in a geometrical sense \cite{Cartan}. When Cartan was examining the representation of groups, he found the mathematical forms of spinors in 1913. Spinors satisfy a linear representation of the groups of rotations of a space of any dimension by Cartan, and spinors are directly concerned with geometry in addition to their relationship with physics \cite{Hladik}. The set of isotropic vectors of the vector space $\mathbb{C}^3$ establishes a two-dimensional surface in the two-dimensional complex space $\mathbb{C}^2$ thanks to the study \cite{Cartan}. Conversely, these vectors in $\mathbb{C}^2$ represent the same isotropic vectors. These vectors are obtained as complex two-dimensional in the space $\mathbb{C}^2$ according to Cartan \cite{Cartan, erisir1}. Further, the triads of unit vectors which are orthogonal by twos were written in terms of a single vector that has two complex components, which is called a spinor \cite{Cartan,Castillobook, delcastillo}. Also, Pauli matrices, which determined the electron spin in quantum theory, were studied by Pauli in \cite{paulii}. One can say that the wave function of an electron can be represented with the help of a vector with two complex components in 1927 by Pauli; this vector is also called a spinor \cite{Hladik}. To get more information concerning the spinors, the following studies can be examined \cite{Vaz,Cartan,Castillobook, delcastillo, Brauer,Lounesto}. In geometric meaning, del Castillo and Barrales investigated the spinor representations of the Serret-Frenet frame \cite{delcastillo}. Then, spinor representations of Bishop frame \cite{kisi}, Darboux frame \cite{unal} and Sabban \cite{senyurt} frame were introduced. Moreover, the spinor equations of involute-evolute curves \cite{erisir2}, Bertrand curves \cite{erisir1}, and successor curves \cite{erisiryeni} were obtained. In addition to these, the hyperbolic spinors were studied, bringing together the different frames \cite{Balci,erisir3,ketenci1,ketenci}.

In another respect, Vivarelli \cite{Vivarelli} determined the relation between the spinors and quaternions,
and according to the relation between quaternions and rotations in Euclidean 3-space, the spinor representations of these three-dimensional rotations were studied \cite{horadamspinor}.  Tarak\c{c}{\i}o\u{g}lu et al. \cite{tarakcioglu} studied the relations between the hyperbolic spinors and split quaternions.

In addition to these, studying the quaternions (real, split, generalized and etc.) with special recurrence sequence components have been examined several researchers.  In the existing literature, lots of studies have been completed with respect to these concepts as follows: bicomplex quaternions with generalized Tribonacci numbers components are introduced by K{\i}z{\i}late\c{s} et al. \cite{canbicomplextribonacci}. Flaut and Shpakivskyi \cite{flaut} determined the generalized Fibonacci quaternions and Fibonacci-Narayana quaternions. Cerda-Morales introduced the real quaternions with generalized Tribonacci numbers components \cite{cerda2}, real quaternions with third-order Jacobsthal numbers components \cite{Moralesidentities}, generalized quaternions with third-order Jacobsthal numbers components \cite{Moralesjacobsthalgeneralizedquaternions}, and third-order $\bar h$-Jacobsthal and third-order $\bar h$-Jacobsthal-Lucas sequences and related quaternions \cite{Moraleshjacobsthal}.
Additionally, Ta\c{s}c{\i} \cite{TasciPadovanquaternion} obtained the real quaternions with Padovan and Pell-Padovan numbers components. G\"unay and Ta\c{s}kara defined some properties of real quaternions with Padovan numbers components in \cite{Gunay}. Also, G\"unay \cite{gunaytez} studied the real quaternions with some generalized third-order recurrence numbers components. Moreover, Di\c{s}kaya and Menken scrutinized the real quaternions with $(s,t)$-Padovan and $(s,t)$-Perrin numbers components \cite{Diskaya1} and split quaternions with $(s,t)$-Padovan and $(s,t)$-Perrin numbers components \cite{Diskaya2}. Further, \.{I}\c{s}bilir and G{\"u}rses determined generalized quaternions with Pell-Padovan numbers components \cite{IsbilirandGurses2} and generalized quaternions with Padovan and Perrin numbers components \cite{IsbilirandGurses1}, generalized quaternions with generalized Jacobsthal numbers components \cite{IsbilirandGurses3}, generalized quaternions with generalized $3$-primes and generalized reverse $3$-primes numbers components \cite{conf-1}. Recently, \.{I}\c{s}bilir and G\"urses determined the generalized quaternions with generalized Tribonacci numbers components \cite{zehranurten}, which are quite big number sequence and include several types numbers such as real quaternions with GTN components, split quaternions with GTN components, semi-quaternions with GTN components, split semi-quaternions with GTN components, 1/4 quaternions with GTN components \cite{zehranurten}.

Lately, Eri\c{s}ir and G\"ung\"or \cite{fibonaccispinor} have determined the Fibonacci and Lucas spinors, and also Eri\c{s}ir \cite{horadamspinor} has examined the Horadam spinors. Then, Di\c{s}kaya and Menken \cite{pps} have introduced the Padovan and Perrin spinors. Also, Kumari et al. \cite{kumari} have investigated the $k$-Fibonacci and $k$-Lucas spinors. \"Oz\c{c}evik and Dertli \cite{ozcevik} determined the hyperbolic Jacobsthal spinors. Moreover, \.{I}\c{s}bilir et al. determined the generalized Tribonacci spinors \cite{zmm} and Padovan and Perrin hyperbolic spinors \cite{zim}. 

This study is organized into 5 sections. In Section \ref{basicconcepts}, we remind some required notions and notations concerning the hyperbolic spinors, generalized Tribonacci numbers, quaternions, and generalized Tribonacci quaternions. In Section \ref{section3}, we define and examine the generalized Tribonacci hyperbolic spinors and some special cases. Then, we investigate some equations such as algebraic properties, mate, conjugates, and some fundamental equations and formulas such as recurrence relation, Binet formula, generating function, exponential generating function, Poisson generating function, sum formulas, and special determinant equalities. Also, we construct some numerical algorithms. Further, we examine a special case of generalized Tribonacci hyperbolic spinors as an example. Also, we give a short introduction for further research: generalized Tribonacci polynomial hyperbolic spinors in Section \ref{sec4}. Then, in Section \ref{conclusion}, we present the conclusions.

\section{Fundamental Terminalogy}\label{basicconcepts}
In this part of this study, we remind some fundamental and required notions and notations with respect to the generalized Tribonacci number and special cases, split quaternions with generalized Tribonacci numbers, and hyperbolic spinors.

\subsection{Split quaternions}
The split quaternion $q \in \mathbb{H}_S$ is written as $q={S_q}+{\overrightarrow V_q}$, where ${S_q} = {q_0}$ is scalar part and ${\overrightarrow V_q} = {q_1}i + {q_2}j + {q_3}k$ is vector part. For the split quaternions $q,p \in \mathbb{H}_S$, some algebraic properties are written as follows \cite{splitfibonacci}:
\begin{itemize}
\item[\ding{93}] \textbf{Addition/Subtraction} $$q \pm p = {q_0} \pm {p_0}+\left( {{q_1}\pm {p_1}}\right)i+\left( {{q_2} \pm {p_2}}\right)j +\left({{q_3} \pm {p_3}}\right)k,$$
\item [\ding{93}]\textbf{Multiplication by a scalar} $$\omega q = \omega{q_0} + \omega {q_1}i + \omega {q_2}j + \omega {q_3}k, \quad \omega \in \mathbb{R}
,$$
\item[\ding{93}] \textbf{Multiplication}
$$\begin{array}{rl}
qp&=\left({q_0} + {q_1}i + {q_2}j + {q_3}k\right)\left({p_0} + {p_1}i + {p_2}j + {p_3}k\right)\\
&=\left( q_0p_0-q_1p_1+q_2p_2+q_3p_3 \right)+\left( q_0p_1-q_1p_0-q_2p_3+q_3p_2 \right)i\\&\,\,\,\,+\left( q_0p_2-q_1p_3+q_2p_0+q_3p_1 \right)j+\left( q_0p_3-q_1p_2-q_2p_1+q_3p_0 \right)k\\
&=S_qS_p+g\left( \overrightarrow{V}_q,\overrightarrow{V}_p \right)+S_q\overrightarrow{V}_p+S_p\overrightarrow{V}_q+\overrightarrow{V}_q\wedge\overrightarrow{V}_p
.
\end{array}$$
where \begin{equation*}
    g\left( \overrightarrow{V}_q,\overrightarrow{V}_p\right)=-q_0p_0+q_1p_1+q_2p_2+q_3p_3
    \end{equation*}
and
\begin{equation*}
\overrightarrow{V}_q\wedge\overrightarrow{V}_p=\begin{vmatrix} \begin{array}{ccc}
    -i & j & k \\
    q_1&q_2&q_3\\
    p_1&p_2&p_3
\end{array}   \end{vmatrix}.
\end{equation*}
\item \textbf{Conjugate} The conjugate of the split quaternion $q$ is $$q^*={q_0}-{q_1}i-{q_2}j-{q_3}k.$$
\item \textbf{Norm:} ${N_q}={qq^*}=q_0^2+q_1^2-q_2^2-q_3^2$.
\end{itemize}

\subsection{Hyperbolic spinors}
Assume that $\Omega$ is an $n\times n$ matrix that is defined on the hyperbolic number system $\mathbb{H}$. $\Omega^\dagger$ defined as transposing and conjugating of $\Omega$, $\Omega^\dagger=\overline{\Omega^t}$, which is an $n\times n$ matrix. If  $\Omega$ is a Hermitian matrix concerning $\mathbb{H}$, then $\Omega^t=\Omega$. If $\Omega$ is an anti-Hermitian matrix with respect to $\mathbb{H}$, then $\Omega^t=-\Omega$. Let $\Omega$ be a Hermitian matrix, the equation $UU^\dagger=U^\dagger U=1$ is valid for $U=e^{j\Omega}$. The set of all $n\times n$ type matrices on $\mathbb{H}$ which satisfy the previous equation establishes a group named hyperbolic unitary group, is denoted by $U(n,\mathbb{H})$. If $\det U=1$, then this type group is represented by $SU(n,\mathbb{H})$ \cite{antonuccio,Balci,erisir3}.

Moreover, the Lorentz group is a group of all Lorentz transformations in the Minkowski space and it is a subgroup of the Poincar\'e group. Then, Poincar\'e group is determined as the group of all isometries in the Minkowski space. The term “orthochronous” is a Lorentz transformation that is kept
in the direction of time. In addition to these, the orthochronous Lorentz group is determined as that rigid transformation of Minkowski 3-space that kept both the direction of time and orientation. If determinant $+1$, then
this subgroup is represented as $SO(1,3)$ \cite{carmeli,Balci,erisir3,ketenci}.

Additionally, there is a homomorphism between the group $SO(1,3)$, which is the group of the rotation along the origin, and $SU(2,\mathbb{H})$, which is the group of the unitary $2\times 2$ type matrix. While the elements of the group $SU(2,\mathbb{H})$ present a fillip to the hyperbolic spinors, the elements of the group $SO(1,3)$ present a fillip to the vectors with three real components in Minkowski space \cite{sattinger,Balci,erisir3,ketenci}.

A hyperbolic spinor with two hyperbolic components as follows:
\begin{equation*}
    \psi=\begin{pmatrix}
    \psi_1\\
    \psi_2
    \end{pmatrix}
\end{equation*}
by using the vectors $a,b,c\in\mathbb{R}_1^3$ such that
\begin{equation}\label{spinor}
\begin{split}
    a+jb&=\psi^t\rho\psi,\\
    c&=-\widehat\psi^t\rho\psi,
    \end{split}
\end{equation}
where “t” denotes the transposition, $\overline\psi$ is the conjugate of $\psi$, $\widehat\psi$ is the mate of $\psi$. Then, the followings are satisfied:
\begin{equation*}
   \widehat\psi=-\begin{pmatrix}
   0&1\\
   -1&0
   \end{pmatrix} \overline\psi=-\begin{pmatrix}
    0&1\\
   -1&0
   \end{pmatrix}\begin{pmatrix}
   \overline\psi_1\\
   \overline\psi_2
   \end{pmatrix}=\begin{pmatrix}
   -\overline\psi_2\\
   \overline\psi_1
   \end{pmatrix}.
\end{equation*}
Also, $2\times2$ hyperbolic symmetric matrices, which are cartesian components for the vector $\varsigma=(\varsigma_1,\varsigma_2,\varsigma_3)$
\begin{equation}\label{matrices}
   \varsigma_1= \begin{pmatrix}
    1&0\\
    0&-1
    \end{pmatrix}, \quad   \varsigma_2= \begin{pmatrix}
    j&0\\
    0&j
    \end{pmatrix}, \quad  \varsigma_3= \begin{pmatrix}
    0&-1\\
    -1&0
    \end{pmatrix}
\end{equation}
are written \cite{Balci,erisir3,tarakcioglu,ketenci,ketenci1}.
Additionally, the ordered triads $\{a,b,c\}, \{b,c,a\}, \linebreak \{c,a,b\}$ correspond to different hyperbolic spinors, and the hyperbolic spinors $\psi$ and $-\psi$ correspond to the same ordered orthogonal basis. For the hyperbolic spinors $\psi$ and $\phi$, the following equations are satisfied: 
\begin{align}
\psi^t\rho\phi&=\phi^t\varsigma\psi,\label{prop2}
\\
    \overline{\psi^t\rho\phi}&=-\widehat\psi^t\varsigma\widehat\phi,\label{prop1}
\\
\widehat{\left(\nu_1\psi+\nu_2\phi\right)}&=\overline \nu_1\widehat\psi+\overline \nu_2\widehat\phi,\label{prop3}
\end{align}
where $\nu_1,\nu_2\in\mathbb{H}$
\cite{Balci,ketenci,erisir3,ketenci1}.
Let $\xi=(\xi_1,\xi_2,\xi_3)\in\mathbb{H}^3$ be an isotropic vector (namely, length of this vector is zero: $\langle \xi,\xi \rangle=0$, $\xi\ne0$) in $\mathbb{R}_1^3$.
According to the above notions and notations, the following equations can be given:
\begin{eqnarray*}
     \xi_1=\psi_1^2-\psi^2_2, \quad  \xi_2=j(\psi^2_1+\psi^2_2), \quad \xi_3=-2\psi_1\psi_2.
\end{eqnarray*}
Moreover, the following equations are satisfied:
\begin{eqnarray}\label{xx}
     \psi_1=\pm\sqrt{\frac{\xi_1+j\xi_2}{2}} \quad \text{and} \quad \psi_2=\pm\sqrt{\frac{-\xi_1+j\xi_2}{2}}.
\end{eqnarray}
Then, $||a||=||b||=||c||=\overline\psi^t\psi$. According to the \eqref{spinor} and \eqref{matrices}, the followings
\begin{equation*}
    \begin{array}{rl}
     \xi_1=\psi^t\rho_1\psi, \quad
       \xi_2=\psi^t\rho_2\psi,\quad
        \xi_3=\psi^t\rho_3\psi
    \end{array}
\end{equation*}
and
\begin{align} \label{zz}
         a+jb&=\left(\psi^2_1-\psi^2_2,j(\psi^2_1+\psi^2_2),-2\psi_1\psi_2\right), \\
         c&= \left(\psi_1\overline\psi_2+\overline\psi_1\psi_2,j(\psi_1\overline\psi_2-\overline\psi_1\psi_2),\left|\psi_1\right|^2-\left|\psi_2\right|^2\right)
\end{align}
can be written \cite{Balci,ketenci,erisir3}.
To get more information about the hyperbolic spinor, the studies \cite{Balci,ketenci,tarakcioglu,erisir3} can be examined, as well.

\subsection{Relations between the hyperbolic spinors and split quaternions}
Tarak\c{c}ıo\u{g}lu et al. \cite{tarakcioglu} and Tarak\c{c}{\i}o\u{g}lu \cite{tarakcioglutez} investigated the relations between the hyperbolic spinors and split quaternions. Let the quaternion $q \in \mathbb{R}$ and the hyperbolic spinor $\psi$ is written, then we get \cite{tarakcioglu,tarakcioglutez}:
\begin{equation}\label{3}
    \begin{split}
    f:&\,\,\mathbb{H}\rightarrow\mathbb{S}\\
        &\,\,q\rightarrow f\left(q_0+q_1i+q_2j+q_3k  \right)=\begin{bmatrix}
        q_{0}+q_{3}j\\
        -q_{1}+q_{2}j
        \end{bmatrix}\equiv \psi_n,
    \end{split}
\end{equation}
where the function $f$ is linear, one-to-one, and onto. Hence, $f\left(q+p\right)=f\left(q\right)+f\left(p\right)$ and $f\left(\omega q\right)=\omega f\left(q\right)$, where $\omega\in\mathbb{R}$ and $ker f=\{0\}$. According to the conjugation of the split quaternion $q$, the following is satisfied \cite{tarakcioglu,tarakcioglutez}:
\begin{equation}
    f(q^*)= f\left(q_0-q_1i-q_2j-q_3 k \right)=\begin{bmatrix}
        q_{0}-jq_{3}\\
        q_{1}-jq_{2}
        \end{bmatrix}\equiv \psi^*_n.
\end{equation}
For more detailed information concerning the relations and representations between the hyperbolic spinors and split quaternions, we refer to the studies \cite{tarakcioglu,tarakcioglutez}.

\subsection{Generalized Tribonacci numbers}
The characteristic equation of generalized Tribonacci numbers, which is given in equation \eqref{recurrencerelationgeneralizedgeneralized Tribonacci} is: 
\begin{equation}\label{chareq}
x^3-rx^2-sx-t=0.
\end{equation}
The roots of equation \eqref{chareq} are:
\begin{equation*}
\left\{\begin{split}
\sigma_1&={r}/{3}+\breve\xi+\breve\gamma, \\ 
\sigma_2&={r}/{3}+\breve\epsilon\breve\xi+\breve\epsilon^2\breve\gamma, \\
\sigma_3&={r}/{3}+\breve\epsilon^2\breve\xi+\breve\epsilon\breve\gamma,
\end{split}\right.
\end{equation*}
where 
\begin{equation*}
\begin{split}
   \left\{ \begin{split}
  \breve\xi&=\sqrt[3]{{\frac{{{r^3}}}{{27}} + \frac{{rs}}{6} + \frac{t}{2} + \sqrt {\breve\Upsilon}}}, \\
\breve\gamma&=\sqrt[3]{{\frac{{{r^3}}}{{27}} + \frac{{rs}}{6} + \frac{t}{2} - \sqrt {\breve\Upsilon}}},\\
\breve\epsilon&=\frac{-1+i\sqrt{3}}{2},\\
\breve\Upsilon&=\frac{{{r^3}t}}{{27}} - \frac{{{r^2}{s^2}}}{{108}} + \frac{{rst}}{6} - \frac{{{s^3}}}{{27}} + \frac{{{t^2}}}{4},
\end{split} \right.
\end{split}
\end{equation*}
and 
$
\sigma_1+\sigma_2+\sigma_3=r,\quad
\sigma_1\sigma_2+\sigma_1\sigma_3+\sigma_2\sigma_3=-s,\quad
\sigma_1\sigma_2\sigma_3=t
$ \cite{cerda2}.

If $\Upsilon > 0$, the equation \eqref{recurrencerelationgeneralizedgeneralized Tribonacci} has one real and two non-real solutions, the latter being conjugate complex.
The Binet formula for generalized Tribonacci numbers is given as \cite{cerda2}:
\begin{equation}\label{generalized Tribonaccibinet}
    {V}_n=\cfrac{\Phi_1\sigma^{n}_1}{(\sigma_1-\sigma_2)(\sigma_1-\sigma_3)}+\cfrac{\Phi_2\sigma^{n}_2}{(\sigma_2-\sigma_1)(\sigma_2-\sigma_3)}+\cfrac{\Phi_3\sigma^{n}_3}{(\sigma_3-\sigma_1)(\sigma_3-\sigma_2)}\raisepunct{,}
\end{equation}
\newpage
where
\begin{equation}\label{etalar}
\begin{split}
   \left\{ \begin{split}
\Phi_1=c-(\sigma_2+\sigma_3)b+\sigma_2\sigma_3a,\\
\Phi_2=c-(\sigma_1+\sigma_3)b+\sigma_1\sigma_3a,\\
\Phi_3=c-(\sigma_1+\sigma_2)b+\sigma_1\sigma_2a.
    \end{split} \right.
    \end{split}
\end{equation}
Moreover, a beneficial way to obtain $nth$ generalized Tribonacci number is by implementing $S$-matrix which is examined in \cite{Shannon,Waddill} and is a generalization of the $R$-matrix. The $S$-matrix is written as follows (see \cite{WaddillandSacks,Kalman}):
\begin{equation*}
   S= \left[ {\begin{array}{*{20}{c}}
r&s&t\\
1&0&0\\
0&1&0
\end{array}} \right].
\end{equation*}   
 Due to the generalized Tribonacci number sequence, which is the most general form of the third-order recurrence sequences, this family contains several special cases due to the given values related to the $r,s,t$ and initial values $a,b,c$. Special cases and some special subgroups of this number sequence can be seen in Table~\ref{somespecialcasesofgtntablenew} and Table~\ref{somespecialcasesofgtn}. Generalized Tribonacci numbers also categorized as related to the $r,s,t$ values in Table~\ref{somespecialcasesofgtntablenew} and are classified both $r,s,t$ and $T_0,T_1,T_2$ in Table~\ref{somespecialcasesofgtn} (see \cite{cerda2, Moralesmatrixrepresentation, Morales2, Morales3, Shannon2, Shannon, Soykanthirdpell, Soykangeneralizedgrahaml, Soykanpellpadovan, Soykangeneralized3-primes, Soykanonfourspecialcases, Soykannarayana, Soykangeneralizedreverse, Soykannewsumformulas, Soykangeneralizedjp, Soykanrst, Soykanexplicit, Soykansumformulas, onlineansiklopedi,cerecedadet,moralescomplex,Cereceda,ShannonandHoradam,ShannonHoradamandAnderson,Shannoniterative,ShannonandWong,Soykanbinomgeneralized3primes,Soykanbinomialtransformofgeneralizedtribonacci,Soykanrecurrence,Soykanbinomialtransformgeneralizedreverse3primes,Soykansum,soykan2, shannon}).

\begin{center}
\begin{table}[h]\caption{A brief classification for generalized Tribonacci numbers}\label{somespecialcasesofgtntablenew}
\centering
\begin{tabular*}{450pt}{@{\extracolsep\fill}|l|l|l|@{\extracolsep\fill}}%
\hline
{Name}&${\{V_n\}=\{V_n(V_0,V_1,V_2;r,s,t)\}}$ & {Recurrence Relation}\\
\hline
\footnotesize{G. Tribonacci (usual)} &\footnotesize{$\{\mathscr{A}_n\}=\{V_n(V_0,V_1,V_2;1,1,1)\}$}
& \footnotesize{$\mathscr{A}_n=\mathscr{A}_{n-1}+\mathscr{A}_{n-2}+\mathscr{A}_{n-3}$}
\\
\hline
\footnotesize{G. Padovan}&\footnotesize{$\{\mathscr{G}_{n}\}=\{V_n(V_0,V_1,V_2;0,1,1)\}$}
& \footnotesize{$\mathscr{G}_n=\mathscr{G}_{n-2}+\mathscr{G}_{n-3}$}\\
\hline
\footnotesize{G. Pell-Padovan}&\footnotesize{$\{\mathscr{M}_{n}\}=\{V_n(V_0,V_1,V_2;0,2,1)\}$}
& \footnotesize{$ \mathscr{M}_n=2\mathcal{M}_{n-2}+\mathscr{M}_{n-3}$}\\
\hline
\footnotesize{G. T. Pell}
 &\footnotesize{$\{\mathscr{S}_{n}\}=\{V_n(V_0,V_1,V_2;2,1,1)\}$}
&\footnotesize{$ \mathscr{S}_n=2\mathscr{S}_{n-1}+\mathscr{S}_{n-2}+\mathscr{S}_{n-3}$}\\
\hline
\footnotesize{G. T. Jacobsthal}&\footnotesize{$\{\mathscr{X}_{n}\}=\{V_n(V_0,V_1,V_2;1,1,2)\}$}
& \footnotesize{$ \mathcal{X}_n=\mathscr{X}_{n-1}+\mathscr{X}_{n-2}+2\mathscr{X}_{n-3}$}\\
\hline
\footnotesize{G. Jacobsthal-Padovan} &\footnotesize{$\{\grave{\boldsymbol{\chi}}_{n}\}=\{V_n(V_0,V_1,V_2;0,1,2)\}$}
& \footnotesize{$ \grave{\boldsymbol{\chi}}_n=\grave{\boldsymbol{\chi}}_{n-2}+2\grave{\boldsymbol{\chi}}_{n-3}$}\\
\hline
\footnotesize{G. Narayana}
&\footnotesize{$\{\grave{\boldsymbol{\vartheta}}_{n}\}=\{V_n(V_0,V_1,V_2;1,0,1)\}$}
&\footnotesize{ $ \grave{\boldsymbol{\vartheta}}_n=\grave{\boldsymbol{\vartheta}}_{n-1}+\grave{\boldsymbol{\vartheta}}_{n-3}$}\\
\hline
\footnotesize{G. 3-primes}
&\footnotesize{$\{\grave{\boldsymbol{\kappa}}_{n}\}=\{V_n(V_0,V_1,V_2;2,3,5)\}$}
& \footnotesize{$ \grave{\boldsymbol{\kappa}}_n=2\grave{\boldsymbol{\kappa}}_{n-1}+3\grave{\boldsymbol{\kappa}}_{n-1}+5\grave{\boldsymbol{\kappa}}_{n-3}$}\\
\hline
\footnotesize{G. Reverse 3-primes}
&\footnotesize{$\{\grave{\boldsymbol{\nabla}}_{n}\}=\{V_n(V_0,V_1,V_2;5,3,2)\}$}
& \footnotesize{$ \grave{\boldsymbol{\nabla}}_n=5\grave{\boldsymbol{\nabla}}_{n-1}+3\grave{\boldsymbol{\nabla}}_{n-1}+2\grave{\boldsymbol{\nabla}}_{n-3}$}\\
\hline
\end{tabular*}
\begin{flushleft}
    \footnotesize{* G.: Generalized, T.: Third-order.}
\end{flushleft}
\end{table}
\end{center}
\normalsize

\begin{center}
\begin{table}[H]\caption{Some special cases of generalized Tribonacci numbers}\label{somespecialcasesofgtn}
\centering
\begin{tabular*}{450pt}{@{\extracolsep\fill}|l|l|l|@{\extracolsep\fill}}%
\hline
\footnotesize{Name}&\footnotesize{$\{V_n\}=\{V_n(V_0,V_1,V_2;r,s,t)\}$} &\footnotesize{Recurrence Relation}\\
\hline
\footnotesize{Tribonacci} &\footnotesize{$\{A_n\}=\{V_n(0,1,1;1,1,1)\} $}&\footnotesize{$ A_n=A_{n-1}+A_{n-2}+A_{n-3}$}\\
\hline
\footnotesize{Tribonacci-Lucas}&\footnotesize{$\{B_n\}=\{V_n(3,1,3;1,1,1)\} $}&\footnotesize{$ B_n=B_{n-1}+B_{n-2}+B_{n-3}$}\\
\hline
\footnotesize{Tribonacci-Perrin}&\footnotesize{$\{C_n\}=\{V_n(3,0,2;1,1,1)\} $}&\footnotesize{$ C_n=C_{n-1}+C_{n-2}+C_{n-3}$}\\
\hline
\footnotesize{M. Tribonacci}&\footnotesize{$\{D_n\}=\{V_n(1,1,1;1,1,1)\} $}&\footnotesize{$ D_n=D_{n-1}+D_{n-2}+D_{n-3}$}\\
\hline
\footnotesize{M. Tribonacci-Lucas}&\footnotesize{$\{E_n\}=\{V_n(4,4,10;1,1,1)\} $}&\footnotesize{$ E_n=E_{n-1}+E_{n-2}+E_{n-3}$}\\
\hline
\footnotesize{A. Tribonacci-Lucas}&\footnotesize{$\{F_n\}=\{V_n(4,2,0;1,1,1)\} $}&\footnotesize{$ F_n=F_{n-1}+F_{n-2}+F_{n-3}$}\\
\hline
\footnotesize{Padovan (Cordonnier)}&\footnotesize{$\{G_n\}=\{V_n(1,1,1;0,1,1)\} $}&\footnotesize{$ G_n=G_{n-2}+G_{n-3}$}\\
\hline
\footnotesize{Perrin (Padovan-Lucas)}&\footnotesize{ $\{H_n\}=\{V_n(3,0,2;0,1,1)\}  $}&\footnotesize{$ H_n=H_{n-2}+H_{n-3}$}\\
\hline
\footnotesize{Van der Laan }& \footnotesize{$\{I_n\}=\{V_n(1,0,1;0,1,1)\}  $}&\footnotesize{$ I_n=I_{n-2}+I_{n-3}$}\\
\hline
\footnotesize{Padovan-Perrin} &\footnotesize{ $\{J_n\}=\{V_n(0,0,1;0,1,1)\}  $}&\footnotesize{$ J_n=J_{n-2}+J_{n-3}$}\\
\hline
\footnotesize{M. Padovan}& \footnotesize{$\{K_n\}=\{V_n(3,1,3;0,1,1)\}  $}&\footnotesize{$ K_n=K_{n-2}+K_{n-3}$}\\
\hline
\footnotesize{A. Padovan}&\footnotesize{ $\{L_n\}=\{V_n(0,1,0;0,1,1)\}  $}&\footnotesize{$ L_n=L_{n-2}+L_{n-3}$}\\
\hline
\footnotesize{Pell-Padovan } &\footnotesize{$\{M_n\}=\{V_n(1,1,1; 0,2,1)\} $}&\footnotesize{$ M_n=2M_{n-2}+M_{n-3}$} \\
\hline
\footnotesize{Pell-Perrin} &\footnotesize{$\{N_n\}=\{V_n(3,0,2; 0,2,1)\} $}&\footnotesize{$ N_n=2N_{n-2}+N_{n-3}$}\\
\hline
\footnotesize{T. Fibonacci-Pell }&\footnotesize{$\{O_n\}=\{V_n(1,0,2; 0,2,1)\} $}&\footnotesize{$ O_n=2O_{n-2}+O_{n-3}$}\\
\hline
\footnotesize{T. Lucas-Pell }&\footnotesize{$\{P_n\}=\{V_n(3,0,4; 0,2,1)\} $}&\footnotesize{$ P_n=2P_{n-2}+P_{n-3}$}\\
\hline
\footnotesize{A. Pell-Padovan } &\footnotesize{$\{R_n\}=\{V_n(0,1,0; 0,2,1)\} $}&\footnotesize{$ R_n=2R_{n-2}+R_{n-3}$} \\
\hline
\footnotesize{T. Pell}&\footnotesize{$\{S_n\}=\{V_n(0,1,2; 2,1,1)\}$}&\footnotesize{$ S_n=2S_{n-1}+S_{n-2}+S_{n-3}$}\\
\hline
\footnotesize{T. Pell-Lucas}&\footnotesize{$\{U_n\}=\{V_n(3,2,6; 2,1,1)\}$}&\footnotesize{$ U_n=2U_{n-1}+U_{n-2}+U_{n-3}$}\\
\hline
\footnotesize{T. modified Pell}&\footnotesize{$\{V_n\}=\{V_n(0,1,1; 2,1,1)\}$}&\footnotesize{$ V_n=2V_{n-1}+V_{n-2}+V_{n-3}$}\\
\hline
\footnotesize{T. Pell-Perrin}&\footnotesize{$\{W_n\}=\{V_n(3,0,2; 2,1,1)\}$}&\footnotesize{$ W_n=2W_{n-1}+W_{n-2}+W_{n-3}$}\\
\hline
\footnotesize{T. Jacobsthal}&\footnotesize{$\{X_n\}=\{V_n(0,1,1; 1,1,2)\}$}&\footnotesize{$ X_n=X_{n-1}+X_{n-2}+2X_{n-3}$}\\
\hline
\footnotesize{T. Jacobsthal-Lucas}&\footnotesize{$\{Y_n\}=\{V_n(2,1,5; 1,1,2)\}$}&\footnotesize{$ Y_n=Y_{n-1}+Y_{n-2}+2Y_{n-3}$}\\
\hline
\footnotesize{M. T. Jacobsthal}&\footnotesize{$\{Z_n\}=\{V_n(3,1,3; 1,1,2)\}$}&\footnotesize{$ Z_n=Z_{n-1}+Z_{n-2}+2Z_{n-3}$}\\
\hline
\footnotesize{T. Jacobsthal-Perrin}&\footnotesize{$\{\Gamma_n\}=\{V_n(3,0,2; 1,1,2)\}$}&\footnotesize{$ \Gamma_n=\Gamma_{n-1}+\Gamma_{n-2}+2\Gamma_{n-3}$}\\
\hline
\footnotesize{Jacobsthal-Padovan}&\footnotesize{$\{\chi_n\}=\{V_n(1,1,1; 0,1,2)\}$}&\footnotesize{$ \chi_n=\chi_{n-2}+2\chi_{n-3}$}\\
\hline
\footnotesize{Jacobsthal-Perrin}&\footnotesize{$\{\Delta_n\}=\{V_n(3,0,2; 0,1,2)\}$}&\footnotesize{$ \Delta_n=\Delta_{n-2}+2\Delta_{n-3}$}\\
\hline
\footnotesize{A. Jacobsthal-Padovan}&\footnotesize{$\{\omega_n\}=\{V_n(0,1,0; 0,1,2)\} $}&\footnotesize{$\omega_n=\omega_{n-2}+2\omega_{n-3}$}\\
\hline
\footnotesize{M. Jacobsthal-Padovan}&\footnotesize{$\{\Omega_n\}=\{V_n(3,1,3; 0,1,2)\} $}&\footnotesize{$\Omega_n=\Omega_{n-2}+2\Omega_{n-3}$}\\
\hline
\footnotesize{Narayana}&\footnotesize{$\{\vartheta_n\}=\{V_n(0,1,1; 1,0,1)\} $}&\footnotesize{$\vartheta_n=\vartheta_{n-1}+\vartheta_{n-3}$}\\
\hline
\footnotesize{Narayana-Lucas}&\footnotesize{$\{\tau_n\}=\{V_n(3,1,1; 1,0,1)\} $}&\footnotesize{$\tau_n=\tau_{n-1}+\tau_{n-3}$}\\
\hline
\footnotesize{Narayana-Perrin}&\footnotesize{$\{\rho_n\}=\{V_n(3,0,2; 1,0,1)\} $}&\footnotesize{$\rho_n=\rho_{n-1}+\rho_{n-3}$}\\
\hline
\footnotesize{3-primes}&\footnotesize{$\{\kappa_n\}=\{V_n(0,1,2; 2,3,5)\} $}&\footnotesize{$\kappa_n=2\kappa_{n-1}+3\kappa_{n-2}+5\kappa_{n-3}$}\\
\hline
\footnotesize{Lucas 3-primes}&\footnotesize{$\{\theta_n\}=\{V_n(3,2,10; 2,3,5)\} $}&\footnotesize{$\theta_n=2\theta_{n-1}+3\theta_{n-2}+5\theta_{n-3}$}\\
\hline
\footnotesize{M. 3-primes}&\footnotesize{$\{\gamma_n\}=\{V_n(0,1,1; 2,3,5)\} $}&\footnotesize{$\gamma_n=2\gamma_{n-1}+3\gamma_{n-2}+5\gamma_{n-3}$}\\
\hline
\footnotesize{Reverse 3-primes}&\footnotesize{$\{\nabla_n\}=\{V_n(0,1,5; 5,3,2)\} $}&\footnotesize{$\nabla _n=5\nabla_{n-1}+3\nabla_{n-2}+2\nabla_{n-3}$}\\
\hline
\footnotesize{Reverse Lucas 3-primes}&\footnotesize{$\{\Lambda_n\}=\{V_n(3,5,31; 5,3,2)\} $}&\footnotesize{$\Lambda_n=5\Lambda_{n-1}+3\Lambda_{n-2}+2\Lambda_{n-3}$}\\
\hline
\footnotesize{Reverse M. 3-primes}&\footnotesize{$\{\phi_n\}=\{V_n(0,1,4; 5,3,2)\} $}&\footnotesize{$\phi _n=5\phi_{n-1}+3\phi_{n-2}+2\phi_{n-3}$}\\
\hline
\end{tabular*}
\vspace{-1mm}
\begin{flushleft}
   \footnotesize{* M.: Modified, A.: Adjusted, T.: Third-order.}
\end{flushleft}
\end{table}
\end{center}
\normalsize

\subsection{Generalized quaternions with generalized Tribonacci numbers components}
The $n^{th}$ generalized quaternion with GTN components is represented by $\breve{V}_n$ and determined as:
\begin{equation}
    \breve{V}_n=V_n+V_{n+1}i+V_{n+2}j+V_{n+3}k
\end{equation}
where $V_n$ is the $n^{th}$ generalized Tribonacci number and $i,j,k$ are the generalized quaternionic units that satisfy the multiplication rules given in Equation \eqref{gq}. 
Also, the following recurrence relation is written:
\begin{equation}
 {\breve{V}}_n=r{ \breve{V}}_{n-1}+s{ \breve{V}}_{n-2}+t{ \breve{V}}_{n-3} \quad \text{for all} \quad n\ge 3,
\end{equation}
where the initial values are
\begin{equation*}
\left\{\begin{split}
\breve{V}_0=&a+bi+cj+(rc+sb+ta)k,\\
\breve{V}_1=&b+ci+(rc+sb+ta)j+\left[\left(r^2+s\right)c+\left(rs+t\right)b+rta\right]k,\\
\breve{V}_2=&c+\left(rc+sb+ta\right)i+\left[\left(r^2+s\right)c+\left(rs+t\right)b+rta\right]j\\
&+\left[\left(r^3+2rs+t\right)c+\left(r^2s+s^2+rt\right)b+\left(r^2t+st\right)a\right]k.
  \end{split}\right.
\end{equation*}
The conjugate of the $n^{th}$ generalized quaternions with GTN components is determined as $\breve{V}^*_{n}=V_{n}- V_{n+1}i- V_{n+2}j- V_{n+3}k$.
Also, \.{I}\c{s}bilir and G{\"u}rses give some fundamental and new properties, formulas, and equations concerning generalized quaternions with GTN components. One can see that if $\alpha=1$ and $\beta=-1$, then split quaternions with GTN components are obtained \cite{zehranurten}.

\section{Generalized Tribonacci Hyperbolic Spinors}\label{section3}
In this part of this study, we introduce a new number system combining the hyperbolic spinors and one of the most popular third-order special recurrence sequence generalized Tribonacci numbers by the generalized Tribonacci split quaternions. Additionally, we obtain some algebraic properties and equalities with respect to conjugations. Also, we give some equations such as the recurrence relation, Binet formula, generating function, exponential generating function, Poisson generating function, summation formulas, and matrix formulas. Then, we obtain the determinant equation for calculating the terms of this special sequence. In addition to these, we construct some numerical algorithms concerning this new number system.

\begin{definition}\label{generalized Tribonaccispinor}
Let $\breve{V}_n$ be the $n{th}$ generalized Tribonacci split quaternion. The set of $n{th}$ generalized Tribonacci split quaternion is denoted by $\mathbb{\breve{V}}$. The following transformation, with the help of the correspondence between the split quaternions and hyperbolic spinors, can be written as follows:
\begin{equation}\label{3.1}
    \begin{split}
    f:\mathbb{\breve{V}}&\rightarrow\mathbb{S}\\
        \breve{V}_n&\rightarrow f\left(V_n+V_{n+1}i+V_{n+2}j+V_{n+3}k \right)=\begin{bmatrix}
        V_{n}+V_{n+3}j\\
        -V_{n+1}+V_{n+2}j
        \end{bmatrix}\equiv \varphi_n
    \end{split}
\end{equation}
where the split quaternionic units $i,j,k$ are satisfied by the rules which are written in the equation \eqref{splitunits}. Since this transformation is linear and one-to-one but not onto, this new type of sequence is named the generalized Tribonacci hyperbolic spinor sequence. It is a linear recurrence sequence and is established with the help of this transformation.
\end{definition}

According to the above Definition \ref{generalized Tribonaccispinor}, we can give the following initial values:
\begin{equation}\label{values}
\begin{split}
\left\{
\begin{split}
        \varphi_0 &= \begin{bmatrix}
       a+\left(rc+sb+ta\right)j\\
        -b+cj
        \end{bmatrix},\\
        \varphi_1 &= \begin{bmatrix}
       b+\left(\left(r^2+s\right)c+\left(rs+t\right)b+rta\right)j\\
        -c+\left(rc+sb+ta\right)j
        \end{bmatrix},\\
        \varphi_2 &= \begin{bmatrix}
      c+  \left(\left(r^3+2rs+t\right)c+\left(r^2s+s^2+rt\right)b+\left(r^2t+st\right)a\right)j\\
        -\left(rc+sb+ta\right)+\left(\left(r^2+s\right)c+\left(rs+t\right)b+rta\right)j
        \end{bmatrix}.
        \end{split}\right.
        \end{split}
    \end{equation}

Let us obtain some algebraic properties related to the generalized Tribonacci hyperbolic spinor sequence, such as addition, subtraction, and multiplication by a scalar, respectively. Let us consider $\varphi_n,\varphi_m\in\mathbb{S}$ for $n,m\ge0$:
\begin{itemize}
    \item[\ding{93}] \textbf{Addition/Subtraction:} 
    \begin{align*}
        \varphi_n \pm \varphi_m = &\begin{bmatrix}
        V_{n}+V_{n+3}j\\
        -V_{n+1}+V_{n+2}j
        \end{bmatrix}\pm\begin{bmatrix}
        V_{m}+V_{n+3}j\\
        -V_{m+1}+V_{m+2}j
        \end{bmatrix}\\
        =&\begin{bmatrix}
        V_{n}\pm V_{m}+\left(V_{n+3}\pm V_{m+3}    \right)j\\
       -\left( V_{n+1}\pm V_{m+1}\right)+\left(V_{n+2}\pm V_{m+2}\right)j
        \end{bmatrix}.
        \end{align*}
    \item[\ding{93}] \textbf{Multiplication by a scalar:} 
    \begin{align*}
        \lambda\varphi_n = \begin{bmatrix}
        \lambda V_{n}+\lambda V_{n+3}j\\
        -\lambda V_{n+1}+\lambda V_{n+2}j
        \end{bmatrix}, \quad \lambda\in \mathbb{R}.
    \end{align*}
    \end{itemize}

\begin{theorem}[\textbf{Recurrence Relation}]
For all $n\ge0$, the following recurrence relation is satisfied for the generalized Tribonacci hyperbolic spinors:
\begin{equation}\label{recurrencerelation}
    \varphi_{n}=r\varphi_{n-1}+s\varphi_{n-2}+t\varphi_{n-3}.
\end{equation}
\end{theorem}
\begin{proof} With the help of the equations \eqref{recurrencerelationgeneralizedgeneralized Tribonacci} and \eqref{3.1}, we finish the proof.
      \begin{equation*}
    \begin{split}
        r\varphi_{n-1}+s\varphi_{n-2}+t\varphi_{n-3}=&r\begin{bmatrix}
         V_{n-1}+ V_{n+2}j\\
        -V_{n}+ V_{n+1}j
        \end{bmatrix}+s\begin{bmatrix}
         V_{n-2}+ V_{n+1}i\\
        -V_{n-1}+ V_{n}j
        \end{bmatrix}\\&+t\begin{bmatrix}
         V_{n-3}+ V_{n}j\\
       - V_{n-2}+ V_{n-1}j
        \end{bmatrix}\\
= &\begin{bmatrix}
         rV_{n-1}+sV_{n-2}+tV_{n-3}+ \left(rV_{n+2}+sV_{n+1}+tV_{n}\right)j\\
        -\left(rV_{n}+sV_{n-1}+tV_{n-2}\right)+ \left(rV_{n+1}+sV_{n}+tV_{n-1}\right)j
        \end{bmatrix}\\
        = &\begin{bmatrix}
         V_{n}+V_{n+3}j\\
        -V_{n+1}+V_{n+2}j
        \end{bmatrix}\\
    =&\varphi_{n}.
     \end{split}
\end{equation*}
\end{proof}
Now, let us construct a numerical algorithm:
\begin{table}[!ht]
\centering
\caption{A numerical algorithm for finding ${n{th}}$ term of the Padovan hyperbolic spinor}\label{t-1}
\begin{tabular}{| l |}
  \hline 	 
  Numerical Algorithm 	\\ \hline	
{\bf{(1)}} Begin\\
{\bf{(2)}} Input $\varphi_0,\varphi_1$ and $\varphi_2$\\
{\bf{(3)}} Form $\varphi_n$ according to the equation \eqref{recurrencerelation}  \\
{\bf{(4)}} Compute $\varphi_n$   \\
{\bf{(5)}} Output  $ \varphi_n\equiv\begin{bmatrix}
        V_{n}+V_{n+3}j\\
        -V_{n+1}+V_{n+2}j
        \end{bmatrix}$	\\
        {\bf{(6)}} Finish
\\ \hline  
\end{tabular}
\end{table}
\newpage
In the following Table \ref{somespecialcasesofgtntablenew2}, we give some special cases (as a group) of generalized Tribonacci hyperbolic spinors.
The other table includes the special cases with respect to both the $r,s,t$ and initial values is constructed according to the Table \ref{somespecialcasesofgtn}, Definition \ref{generalized Tribonaccispinor} and equation \eqref{recurrencerelation}. That is, recurrence relations and definitions for special cases with respect to both $r,s,t$ and initial values can be written similarly, as in the Table \ref{somespecialcasesofgtntablenew2}.
  \begin{center}
\begin{table}[H]
\centering
\caption{Some Classification for Generalized Tribonacci Hyperbolic  Spinors}\label{somespecialcasesofgtntablenew2}
\begin{tabular*}{450pt}{@{\extracolsep\fill}|l|l|l|@{\extracolsep\fill}}%
\hline
Name & Definition& Recurrence Relation\\
\hline
G. Tribonacci (usual) H.S. &$
\eta_n=\begin{bmatrix}
        \mathscr{A}_{n}+\mathscr{A}_{n+3}j\\
       -\mathscr{A}_{n+1}+\mathscr{A}_{n+2}j
        \end{bmatrix} $&$\eta_n=\eta_{n-1}+\eta_{n-2}+\eta_{n-3}$
        
\\
\hline
G. Padovan H.S.&$\delta_n=\begin{bmatrix}
        \mathscr{G}_{n}+\mathscr{G}_{n+3}j\\
        -\mathscr{G}_{n+1}+\mathscr{G}_{n+2}j
        \end{bmatrix}$& $\delta_n=\delta_{n-2}+\delta_{n-3}$       
\\
        \hline
G. Pell-Padovan H.S. &$\epsilon_n=\begin{bmatrix}
        \mathscr{M}_{n}+\mathscr{M}_{n+3}j\\
       - \mathscr{M}_{n+1}+\mathscr{M}_{n+2}j
        \end{bmatrix}$&$ \epsilon_n=2\epsilon_{n-2}+\epsilon_{n-3}$                         
\\
\hline
    
G. T. Pell H.S.
 &$\flat_n=\begin{bmatrix}
        \mathscr{S}_{n}+\mathscr{S}_{n+3}j\\
        -\mathscr{S}_{n+1}+\mathscr{S}_{n+2}j
        \end{bmatrix}$&$\flat_n=2\flat_{n-1}+\flat_{n-2}+\flat_{n-3} $          
\\
\hline
G. T. Jacobsthal H.S. &$\Omega_n=\begin{bmatrix}
        \mathscr{X}_{n}+\mathscr{X}_{n+3}j\\
       - \mathscr{X}_{n+1}+\mathscr{X}_{n+2}j
        \end{bmatrix}$&$\Omega_n=\Omega_{n-1}+\Omega_{n-2}+2\Omega_{n-3} $        
\\
\hline
G. Jacobsthal-Padovan H.S. &$\upsilon_n= \begin{bmatrix}
        \grave{\boldsymbol{\chi}}_{n}+\grave{\boldsymbol{\chi}}_{n+3}j\\
       - \grave{\boldsymbol{\chi}}_{n+1}+\grave{\boldsymbol{\chi}}_{n+2}j
        \end{bmatrix}$&$ \upsilon_n=\upsilon_{n-2}+2\upsilon_{n-3}$      
\\
         \hline
G. Narayana H.S.
&$\Upsilon_{n}=\begin{bmatrix}
    \grave{\boldsymbol{\vartheta}}_{n}+\grave{\boldsymbol{\vartheta}}_{n+3}j\\
        -\grave{\boldsymbol{\vartheta}}_{n+1}+\grave{\boldsymbol{\vartheta}}_{n+2}j
        \end{bmatrix} $&$\Upsilon_n=\Upsilon_{n-1}+\Upsilon_{n-3}$
               
\\
\hline
G. 3-primes H.S.
&$\phi_{n}=\begin{bmatrix}
    \grave{\boldsymbol{\kappa}}_{n}+\grave{\boldsymbol{\kappa}}_{n+3}j\\
        -\grave{\boldsymbol{\kappa}}_{n+1}+\grave{\boldsymbol{\kappa}}_{n+2}j
        \end{bmatrix} $&$\phi_n=2\varphi_{n-1}+3\phi_{n-2}+5\phi_{n-3} $        
\\
        \hline
G. Reverse 3-primes H.S. & $\nu_{n}=\begin{bmatrix}
    \grave{\boldsymbol{\nabla}}_{n}+\grave{\boldsymbol{\nabla}}_{n+3}j\\
        -\grave{\boldsymbol{\nabla}}_{n+1}+\grave{\boldsymbol{\nabla}}_{n+2}j
        \end{bmatrix} $&$\tau_{n}=5\tau_{n-1}+3\tau_{n-2}+2\tau_{n-3} $      
\\
\hline
\end{tabular*}
\begin{flushleft}
* G.: Generalized, T.:Third-order, H.S.: Hyperbolic spinor
\end{flushleft}
\end{table}
\end{center}

\begin{definition} Let the conjugate of the $n{th}$ generalized Tribonacci split quaternion is denoted by $\breve{V}^*_{n}=V_{n}- V_{n+1}i- V_{n+2}j- V_{n+3}k$. The followings hold:
\begin{itemize}
    \item [\ding{93}]
The $n{th}$ generalized Tribonacci hyperbolic spinor $\varphi_n^*$ corresponding to the conjugate of the $n{th}$ generalized Tribonacci split quaternion is expressed by:
\begin{equation*}
f\left(\breve{V}_{n}^{*}\right)=f\left(V_{n}-V_{n+1}i-V_{n+2}j- V_{n+3}k\right)=\left[\begin{array}{l}
V_{n}-V_{n+3}j  \\
V_{n+1}-V_{n+2}j
\end{array}\right] \equiv \varphi_{n}^{*}.
\end{equation*}
\item [\ding{93}]
The matrix
$C=\left[\begin{array}{cc}0 & 1 \\
-1 & 0\end{array}\right]$ is given. 
The ordinary hyperbolic conjugate of $n{th}$ generalized Tribonacci hyperbolic spinor $\varphi_n$ is obtained as follows:
\begin{equation*}
\overline{\varphi}_n=\left[\begin{array}{c}
V_n-V_{n+3}j \\
-V_{n+1}- V_{n+2}j
\end{array}\right]
\end{equation*}
\item[\ding{93}]
Hyperbolic conjugate of generalized Tribonacci hyperbolic spinor \linebreak $\tilde{\varphi}_{n}=jC \overline{\varphi}_{n}$ of $n{th}$ generalized Tribonacci hyperbolic spinor $\varphi_{n}$ is written as follows:
\begin{equation*}
\tilde{\varphi}_n=\left[\begin{array}{l}
-V_{n+2}- V_{n+1}j \\
V_{n+3}- V_nj
\end{array}\right],
\end{equation*}
where by using the study of Cartan \cite{Cartan}.
\item [\ding{93}]
The hyperbolic mate of $n{th}$ generalized Tribonacci hyperbolic spinor \linebreak  $\check{\varphi}_{n}=-C \overline\varphi_{n}$ is
\begin{equation*}
\check{\varphi}_n=\left[\begin{array}{l}
V_{n+1}+ V_{n+2}j \\
V_n-V_{n+3}j 
\end{array}\right],
\end{equation*}
where by using the study of del Castillo and Barrales \cite{delcastillo}.
\end{itemize}
\end{definition}
The following Theorems \ref{th-2}-\ref{th-8-1} are written without proofs because the proofs are straightforward with the help of the matrix $C$ and the conjugation properties of generalized Tribonacci hyperbolic spinors. 
\begin{theorem}\label{th-2}
The following properties are satisfied.
\begin{multicols}{3}
\begin{itemize}
\item [\ding{93}] $ 
 \overline{\varphi}_n=C \check{\varphi}_{n}$
\item [\ding{93}]$ 
\,  \check{\varphi}_{n}=-j \tilde{\varphi}_{n}$
\item [\ding{93}] $
  \overline{\varphi}_{n}=-j C \tilde{\varphi}_{n}$
\end{itemize}
 \end{multicols}
\end{theorem}

\begin{theorem}\label{th-3}
 The following properties are satisfied.
    \begin{multicols}{2}
    \begin{itemize}
        \item [\ding{93}] $
\varphi_n+\varphi_n^*=\begin{bmatrix}
            2V_n\\
            0
        \end{bmatrix}$
        \end{itemize}
        \columnbreak
        \begin{itemize}
        \item [\ding{93}] $\varphi_n-\varphi_n^*=2\begin{bmatrix}
            V_{n+3}j\\
            -V_{n+1}+V_{n+2}j
        \end{bmatrix}$
\end{itemize}
 \end{multicols}
\begin{multicols}{2}
\begin{itemize}
 \item [\ding{93}] $
 \,   \varphi_n+\overline{\varphi}_n=2\begin{bmatrix}
            V_n\\
              - V_{n+1}
          \end{bmatrix}$
            \end{itemize}
        \columnbreak
        \begin{itemize}
        \item [\ding{93}] $
\varphi_n-\overline{\varphi}_n=2j\begin{bmatrix}
              V_{n+3}\\
              V_{n+2}
          \end{bmatrix}$
\end{itemize}
\end{multicols}
\end{theorem}

\begin{theorem}\label{th-5-1}
The following properties are satisfied.
    \begin{itemize}
    \item [\ding{93}]  $
 \varphi_n+\widetilde{\varphi}_n=\begin{bmatrix}
                -V_{n-1}+V_{n}j\\
                V_{n}+V_{n-1}j
            \end{bmatrix}$
\item  [\ding{93}] $  \varphi_n+\widetilde{\varphi}_n=\begin{bmatrix}
              V_{n}+V_{n+2}+\left(V_{n+3}+V_{n+1}\right)j\\
               -V_{n+1}-V_{n+3}+\left(V_{n+2}+V_{n}   \right)j
           \end{bmatrix}$
        \end{itemize}
    \begin{itemize}
\item  [\ding{93}] $  \varphi_n+\check{\varphi}_n=\begin{bmatrix}
               V_{n+3}+V_{n+5}j\\
               -V_{n+1}+V_n+\left(V_{n+2}-V_{n+3}   \right)j
           \end{bmatrix}$
    \item [\ding{93}]  $
 \varphi_n-\check{\varphi}_n=\begin{bmatrix}
                V_{n}-V_{n+1}+\left(V_{n+3}-V_{n+2}\right)j\\
                -V_{n+3}+V_{n+5}j
            \end{bmatrix}$
        \end{itemize}
\end{theorem}

\begin{theorem}\label{th-6}
The following properties are satisfied.
\begin{multicols}{2}
\begin{itemize}
        \item  [\ding{93}] $ \varphi_n^*+\overline{\varphi}_n=2\begin{bmatrix}
                V_{n}-V_{n+3}j\\
                V_{n+2}j
            \end{bmatrix}$
        \end{itemize}
        \columnbreak
        \begin{itemize}
\item [\ding{93}]  $\varphi_n^*-\overline{\varphi}_n=2\begin{bmatrix}
                0\\
                V_{n+1}
            \end{bmatrix}$
        \end{itemize}
        \end{multicols}
\end{theorem}

\begin{theorem}\label{th-6-1}
The following properties are satisfied.
\begin{itemize}
        \item  [\ding{93}] $  \varphi_n^*+\widetilde{\varphi}_n=\begin{bmatrix}
                V_{n}-V_{n+2}-\left(V_{n+3}+V_{n+1}\right)j\\
                V_{n+1}+V_{n+3}-\left(V_{n+2}+V_{n}\right)j
            \end{bmatrix}$
        \item  [\ding{93}] $  \varphi_n^*-\widetilde{\varphi}_n=\begin{bmatrix}
                V_{n}+V_{n+2}-\left(V_{n+3}-V_{n+1}\right)j\\
                V_{n+1}-V_{n+3}-\left(V_{n+2}-V_{n}\right)j
            \end{bmatrix}$
        \end{itemize}
    \begin{itemize}
                        \item [\ding{93}]      $ \varphi_n^*+\check{\varphi}_n=\begin{bmatrix}
                V_{n+3}+V_{n+5}j\\
                V_{n+3}-V_{n+5}j
            \end{bmatrix}$
    
\item [\ding{93}]     $   \varphi_n^*-\check{\varphi}_n=\begin{bmatrix}
                V_{n}-V_{n+1}-\left(V_{n+3}+V_{n+2}\right)j\\
               V_{n+1}-V_{n}-\left(V_{n+3}-V_{n+2}\right)j
            \end{bmatrix}$

\end{itemize}
\end{theorem}

\begin{theorem}\label{th-8-1}
The following properties are satisfied.
\begin{itemize}
\item [\ding{93}]$
\overline{\varphi}_n+\widetilde{\varphi}_n=\begin{bmatrix}
                -V_{n-1}-\left(V_{n+3}+V_{n+1}\right)j\\
               -V_{n}-\left(V_{n+2}+V_n\right)j
            \end{bmatrix}$

 \item   [\ding{93}] $
   \overline{\varphi}_n-\widetilde{\varphi}_n=\begin{bmatrix}
                V_{n}+V_{n+2}-V_{n}j\\
               -V_{n+1}-V_{n+3}-V_{n-1}j
            \end{bmatrix}$

\end{itemize}
\begin{itemize}
\item [\ding{93}]$
\widetilde{\varphi}_n+\check{\varphi}_n=\begin{bmatrix}
                V_{n+1}-V_{n+2}+\left(V_{n+2}-V_{n+1}\right)j\\
               V_{n+3}+V_n-V_{n+1}j
            \end{bmatrix}$

 \item  [\ding{93}] $
   \widetilde{\varphi}_n-\check{\varphi}_n=\begin{bmatrix}
                -V_{n+4}-V_{n+4}j\\
               V_{n+1}-V_{n+1}j
            \end{bmatrix}$

\end{itemize}
\begin{itemize}
\item [\ding{93}]$
 \overline{\varphi}_n-\check{\varphi}_n=\begin{bmatrix}
                -V_{n+4}-V_{n+4}j\\
               V_{n+1}-V_{n+1}j
            \end{bmatrix}$

 \item  [\ding{93}]$
   \overline{\varphi}_n+\check{\varphi}_n=\begin{bmatrix}
                -V_{n+2}+V_{n+1}+\left(V_{n+2}-V_{n+1}\right)j\\
               V_{n+3}+V_{n}-\left(V_{n}+V_{n+3}\right)j
            \end{bmatrix}$

\end{itemize}
\end{theorem}

\begin{definition} [\textbf{Norm}] The norm of the $n{th}$ generalized Tribonacci split quaternion $N\left(\breve{V}_n\right)=\breve{V}_n\breve{V}_n^*$ is equal to the norm of the associated generalized Tribonacci hyperbolic spinors:
\begin{equation*}
N(\breve{V}_n)=\overline\varphi_n^t\varphi_n,
\end{equation*}
and with the help of the Theorem \ref{th-2}, we write the norm of the generalized Tribonacci hyperbolic spinors as follows:
\begin{equation*}
N\left(\varphi_n\right)=\overline{\varphi}_n^{t}\varphi_n.
\end{equation*}
\end{definition}

\begin{theorem}[\textbf{Generating Function}]
Let $\varphi_n$ be the $n{th}$ generalized Tribonacci hyperbolic spinor. For all $n\ge0$, the following generating function is satisfied:
\begin{align}
    \sum\limits_{n = 0}^\infty {\varphi}_n {x^n} &= \cfrac{{{{{ \varphi}_0}} + ({{ \varphi}_1}-r{ \varphi}_0)x + ({{{ \varphi}_2}}- r{ \varphi}_1-s{ \varphi}_0){x^2}}}{{1-rx-s{x^2}-t{x^3}}}\raisepunct{.}\label{generatingfunction1}
\end{align}
where
\begin{equation*}
\left\{
\begin{split}
&{{{ \varphi}_0}}=\begin{bmatrix}
        a+\left( rc+sb+ta  \right)j\\
       -b+cj
    \end{bmatrix},
    \\
&{{{ \varphi}_1}}- r{ \varphi}_0=\begin{bmatrix}
        b-ar+\left( sc+tb \right)j\\
        -c+rb+\left(sb+ta \right)j
    \end{bmatrix},
    \\
& {{{ \varphi}_2}}- r{ \varphi}_1-s{ \varphi}_0=   \begin{bmatrix}
 a+\left( rc+sb+ta \right)j\\
        -b+cj
    \end{bmatrix}.
    \end{split}
    \right.
\end{equation*}
\end{theorem}
\begin{proof}
Assume that the following equation \eqref{a} is the generating function for the generalized Tribonacci hyperbolic spinor.
\begin{equation}\label{a}
 \sum\limits_{n = 0}^\infty \varphi_n {x^n}  = {\varphi}_0 + {{ \varphi}_1}x + {{ \varphi}_2}{x^2} + \ldots + {{ \varphi}_n}{x^n} + \ldots  
 \end{equation}
If we multiply both sides of equation \eqref{a} by $rx,s{x^2},t{x^3}$, we get:
\begin{equation*}
\begin{array}{rl}
 r{x} \sum\limits_{n = 0}^\infty {\varphi}_n {x^n}& = r{{ \varphi}_0}{x} + r{{ \varphi}_1}{x^2} + r{{ \varphi}_2}{x^3} + \ldots +r{{ \varphi}_{n }}{x^{n+1}} + \ldots\\
    s{x^2} \sum\limits_{n = 0}^\infty {\varphi}_n {x^n}& = s{{ \varphi}_0}{x^2} + s{{ \varphi}_1}{x^3} + s{{ \varphi}_2}{x^4} +  \ldots +s{{ \varphi}_{n}}{x^{n+2}} + \ldots\\
    {tx^3} \sum\limits_{n = 0}^\infty {\varphi}_n {x^n} & = t{{ \varphi}_0}{x^3} + t{{ \varphi}_1}{x^4} + t{{ \varphi}_2}{x^5} + \ldots + t{{ \varphi}_{n}}{x^{n+3}} + \ldots
    \end{array}
\end{equation*}
With the help of the equation \eqref{recurrencerelation}, we obtain:
\begin{equation*}
\left(1-rx-s{x^2}-t{x^3}\right)\sum\limits_{n = 0}^\infty {\varphi}_n {x^n}= {{ \varphi}_0} + ({{ \varphi}_1}-r{{ \varphi}_0})x + ({{ \varphi}_2} - r{{ \varphi}_1}-s{{ \varphi}_0}){x^2}.
\end{equation*}
Consequently, we attain the equation \eqref{generatingfunction1}.
\end{proof}

\begin{theorem}[\textbf{Binet Formula}] Let $\varphi_n$ be the $n{th}$ generalized Tribonacci hyperbolic spinor. The following Binet formula is written:
\begin{equation}\label{generalized Tribonaccispinorbinet}
\begin{split}
   \varphi_n
        =\cfrac{\zeta_1\Phi_1\sigma^{n}_1}{(\sigma_1-\sigma_2)(\sigma_1-\sigma_3)}+\cfrac{\zeta_2\Phi_2\sigma^{n}_2}{(\sigma_2-\sigma_1)(\sigma_2-\sigma_3)}+\cfrac{\zeta_3\Phi_3\sigma^{n}_3}{(\sigma_3-\sigma_1)(\sigma_3-\sigma_2)}\raisepunct{,}
        \end{split}
\end{equation}
where
\begin{equation*}
\zeta_1=\begin{bmatrix}
        1+\sigma^{3}_1j\\
       \left(-1+\sigma_{1}j\right)\sigma_1
        \end{bmatrix},\quad
\zeta_2=\begin{bmatrix}
        1+\sigma^{3}_2j\\
       \left(-1+\sigma_{2}i\right)\sigma_2
        \end{bmatrix}, \quad
\zeta_3=\begin{bmatrix}
       1+\sigma^{3}_3j\\
       \left(-1+\sigma_{3}j\right)\sigma_3
        \end{bmatrix}.
\end{equation*}
\end{theorem}
\begin{proof}
By using the equations \eqref{generalized Tribonaccispinor} and \eqref{generalized Tribonaccibinet}, we obtain:
\begin{equation*}
\begin{split}
  \varphi_n =&\begin{bmatrix}
      \left\{ \begin{split}&
            \frac{\Phi_1\sigma^{n}_1}{(\sigma_1-\sigma_2)(\sigma_1-\sigma_3)}+\frac{\Phi_2\sigma^{n}_2}{(\sigma_2-\sigma_1)(\sigma_2-\sigma_3)}+\frac{\Phi_3\sigma^{n}_3}{(\sigma_3-\sigma_1)(\sigma_3-\sigma_2)}\\&+\left(\frac{\Phi_1\sigma^{n+3}_1}{(\sigma_1-\sigma_2)(\sigma_1-\sigma_3)}+\frac{\Phi_2\sigma^{n+3}_2}{(\sigma_2-\sigma_1)(\sigma_2-\sigma_3)}+\frac{\Phi_3\sigma^{n+3}_3}{(\sigma_3-\sigma_1)(\sigma_3-\sigma_2)}  \right)j\end{split}\right\}\\
     \left\{ \begin{split} &-\left(  \frac{\Phi_1\sigma^{n+1}_1}{(\sigma_1-\sigma_2)(\sigma_1-\sigma_3)}+\frac{\Phi_2\sigma^{n+1}_2}{(\sigma_2-\sigma_1)(\sigma_2-\sigma_3)}+\frac{\Phi_3\sigma^{n+1}_3}{(\sigma_3-\sigma_1)(\sigma_3-\sigma_2)}\right)\\&+\left( \frac{\Phi_1\sigma^{n+2}_1}{(\sigma_1-\sigma_2)(\sigma_1-\sigma_3)}+\frac{\Phi_2\sigma^{n+2}_2}{(\sigma_2-\sigma_1)(\sigma_2-\sigma_3)}+\frac{\Phi_3\sigma^{n+2}_3}{(\sigma_3-\sigma_1)(\sigma_3-\sigma_2)}  \right)j
       \end{split}\right\} \end{bmatrix}\vspace{1mm}\\
  =& \begin{bmatrix}
        \cfrac{\left(1+\sigma^{3}_1j\right)\Phi_1\sigma^{n}_1}{(\sigma_1-\sigma_2)(\sigma_1-\sigma_3)}+\cfrac{\left(1+\sigma^{3}_2j\right)\Phi_2\sigma^{n}_2}{(\sigma_2-\sigma_1)(\sigma_2-\sigma_3)}+\cfrac{\left(1+\sigma^{3}_3j\right)\Phi_3\sigma^{n}_3}{(\sigma_3-\sigma_1)(\sigma_3-\sigma_2)}\\
       \cfrac{\left(1+\sigma_{1}j\right)\Phi_1\sigma^{n+1}_1}{(\sigma_1-\sigma_2)(\sigma_1-\sigma_3)}+\cfrac{\left(1+\sigma_{2}j\right)\Phi_2\sigma^{n+1}_1}{(\sigma_2-\sigma_1)(\sigma_2-\sigma_3)}+\cfrac{\left(1+\sigma_3j\right)\Phi_3\sigma^{n+1}_3}{(\sigma_3-\sigma_1)(\sigma_3-\sigma_2)}
        \end{bmatrix}
        \vspace{1mm}\\
       =& \begin{bmatrix}
        1+\sigma^{3}_1j\\
       \left(-1+\sigma_{1}j\right)\sigma_1
        \end{bmatrix}\frac{\Phi_1\sigma^{n}_1}{(\sigma_1-\sigma_2)(\sigma_1-\sigma_3)}
    +\begin{bmatrix}
        1+\sigma^{3}_2j\\
       \left(-1+\sigma_{2}j\right)\sigma_2
        \end{bmatrix}\frac{\Phi_2\sigma^{n}_2}{(\sigma_2-\sigma_1)(\sigma_2-\sigma_3)}
       \\& +\begin{bmatrix}
        1+\sigma^{3}_3j\\
       \left(-1+\sigma_{3}j\right)\sigma_3
        \end{bmatrix}\frac{\Phi_3\sigma^{n}_3}{(\sigma_3-\sigma_1)(\sigma_3-\sigma_2)}
        \vspace{1mm}
        \\
        =&\frac{\zeta_1\Phi_1\sigma^{n}_1}{(\sigma_1-\sigma_2)(\sigma_1-\sigma_3)}+\frac{\zeta_2\Phi_2\sigma^{n}_2}{(\sigma_2-\sigma_1)(\sigma_2-\sigma_3)}+\frac{\zeta_3\Phi_3\sigma^{n}_3}{(\sigma_3-\sigma_1)(\sigma_3-\sigma_2)}\raisepunct{.}
        \end{split}
\end{equation*}

\end{proof}

\begin{theorem}[\textbf{Exponential Generating Function}]
Let $\varphi_n$ be $n{th}$ generalized Tribonacci hyperbolic spinor. For all $n\ge0$, the following exponential generating function is satisfied for generalized Tribonacci hyperbolic spinor:
\begin{equation}\label{egf}
   \sum\limits_{n = 0}^\infty {{{\varphi }_n}}\cfrac{y^n}{n!}
        =\cfrac{\zeta_1\Phi_1e^{\sigma_1y}}{(\sigma_1-\sigma_2)(\sigma_1-\sigma_3)}+\cfrac{\zeta_2\Phi_2e^{\sigma_1y}}{(\sigma_2-\sigma_1)(\sigma_2-\sigma_3)}+\cfrac{\zeta_3\Phi_3e^{\sigma_1y}}{(\sigma_3-\sigma_1)(\sigma_3-\sigma_2)}\raisepunct{.}
\end{equation}
\end{theorem}
\begin{proof}
Via the equation \eqref{generalized Tribonaccispinorbinet}, we get:
\begin{equation*}
    \begin{split}
       \sum\limits_{n = 0}^\infty {{{\varphi }_n}}\cfrac{y^n}{n!}=&\sum\limits_{n = 0}^\infty\left[
\cfrac{\zeta_1\Phi_1\sigma^{n}_1}{(\sigma_1-\sigma_2)(\sigma_1-\sigma_3)}+\cfrac{\zeta_2\Phi_2\sigma^{n}_2}{(\sigma_2-\sigma_1)(\sigma_2-\sigma_3)}+\cfrac{\zeta_3\Phi_3\sigma^{n}_3}{(\sigma_3-\sigma_1)(\sigma_3-\sigma_2)}
\right]\cfrac{y^n}{n!}\\
      =& \cfrac{\zeta_1\Phi_1}{(\sigma_1-\sigma_2)(\sigma_1-\sigma_3)}\sum\limits_{n = 0}^\infty\cfrac{\left(\sigma_1 y \right)^n}{n!}+\cfrac{\zeta_2\Phi_2}{(\sigma_2-\sigma_1)(\sigma_2-\sigma_3)}\sum\limits_{n = 0}^\infty\cfrac{\left(\sigma_2 y \right)^n}{n!}\\&+\cfrac{\zeta_3\Phi_3}{(\sigma_3-\sigma_1)(\sigma_3-\sigma_2)}\sum\limits_{n = 0}^\infty\cfrac{\left(\sigma_3 y \right)^n}{n!}\\
      =&\cfrac{\zeta_1\Phi_1e^{\sigma_1y}}{(\sigma_1-\sigma_2)(\sigma_1-\sigma_3)}+\cfrac{\zeta_2\Phi_2e^{\sigma_1y}}{(\sigma_2-\sigma_1)(\sigma_2-\sigma_3)}+\cfrac{\zeta_3\Phi_3e^{\sigma_1y}}{(\sigma_3-\sigma_1)(\sigma_3-\sigma_2)}\raisepunct{.}
    \end{split}
\end{equation*}
\end{proof}

\begin{theorem}[\textbf{Poisson Generating Function}]
   Let $\varphi_n$ be the $n{th}$ generalized Tribonacci hyperbolic spinor.
The Poisson generating function is expressed as:
\begin{equation*}
  e^{-y}\sum\limits_{n = 0}^\infty {{{\varphi }_n}}\cfrac{y^n}{n!} =\cfrac{\zeta_1\Phi_1e^{\sigma_1y}}{e^y(\sigma_1-\sigma_2)(\sigma_1-\sigma_3)}+\cfrac{\zeta_2\Phi_2e^{\sigma_1y}}{e^y(\sigma_2-\sigma_1)(\sigma_2-\sigma_3)}+\cfrac{\zeta_3\Phi_3e^{\sigma_1y}}{e^y(\sigma_3-\sigma_1)(\sigma_3-\sigma_2)}\raisepunct{.}
\end{equation*}
\end{theorem}

\begin{proof}
By the equation \eqref{egf}, we have the desired equation since Poisson generating function is expressed as multiplying the exponential generating function by $e^{-y}$ (cf. also \cite{senturk-new}).
\end{proof}

\begin{theorem}\label{matrix} 
Let $\varphi_n$ be the $n{th}$ generalized Tribonacci hyperbolic spinor.
For all $n > 0$, the following matrix equation holds:
   \begin{equation*}
  \left[ {\begin{array}{*{20}{c}}
{ \varphi}_{n+2}\\
{\varphi}_{n+1}\\
{\varphi}_{n}
\end{array}} \right]= \left[ {\begin{array}{*{20}{c}}
r&s&t\\
1&0&0\\
0&1&0
\end{array}} \right]^n  {\left[ {\begin{array}{*{20}{c}}
{ \varphi}_{2}\\
{\varphi}_{1}\\
{\varphi}_{0}
\end{array}} \right]}.
   \end{equation*}
\end{theorem}
\begin{proof}
By mathematical induction, this proof is finished.
\end{proof}

According to Soykan \cite{Soykansumformulas}, we obtain some summation formulas concerning the generalized Tribonacci hyperbolic spinors in Theorem \ref{sum1}. Due to the proof being straightforward, we skip this.
\begin{theorem}\label{sum1}
Let $\varphi_n$ be the $n^{th}$ generalized Tribonacci spinor. For all $m,n>0$, the following summation formulas are written:
    \begin{itemize}
    \item [\ding{93}] $\sum\limits_{n=0}^{m}\varphi_n=\cfrac{ \begin{array}{l}\varphi_{m+3}+(1-r)\varphi_{m+2}+(1-r-s)\varphi_{m+1}-\varphi_{2} +(r-1)\varphi_{1}\\+(r+s-1)\varphi_{0}\end{array}}{(r+s+t-1)}\raisepunct{,}$
       \item [\ding{93}] $\sum\limits_{n=0}^{m}\varphi_{2n}=\cfrac{\begin{array}{l}
       (1-s)\varphi_{2m+2}+(t+rs)\varphi_{2m+1}+(t^2+rt)\varphi_{2m}+(s-1)\varphi_{2}\\+(-t-rs)\varphi_{1}+(r^2-s^2+rt+2s-1)\varphi_{0}\end{array}}{(r+s+t-1)(r-s+t+1)}\raisepunct{,}$
         \item [\ding{93}] $\sum\limits_{n=0}^{m}\varphi_{2n+1}=\cfrac{\begin{array}{l}(r+t)\varphi_{2m+2}+(s-s^2+t^2+rt)\varphi_{2m+1}+(t-st)\varphi_{2m}\\+(-r-t)\varphi_{2}+(-1+s+r^2+rt)\varphi_{1}+(-t+st)\varphi_{0}\end{array}}{(r-s+t+1)(r+s+t-1)}\raisepunct{,}$
  \end{itemize}    
where denominators are not zero.
\end{theorem}
\begin{flushleft}
\textbf{Special Cases:} If $s=1$, we get the following summation formulas:
\begin{itemize}
 \item[\ding{93}] $\sum\limits_{n=0}^{m}\varphi_{2n}=\cfrac{\varphi_{2m+1}+t\varphi_{2m}-\varphi_{1}+r\varphi_{0}}{r+t}\raisepunct{,}$
 \item[\ding{93}] $\sum\limits_{n=0}^{m}\varphi_{2n+1}=\cfrac{\varphi_{2m+2}+t\varphi_{2m+1}-\varphi_{2}+r\varphi_{1}}{r+t}\raisepunct{,}$
\end{itemize}
where denominators are not zero.
\end{flushleft}

By inspiring K\i z\i late\c{s} et al. \cite{canbicomplextribonacci}, we construct the following Theorem \ref{det} to compute the $n{th}$ term of generalized Tribonacci hyperbolic spinor.

\begin{theorem}\label{det} Let $\varphi_n$ be the $n{th}$ generalized Tribonacci hyperbolic spinor. For all $n \ge 0$, the following equation is obtained.
\begin{equation}\label{det1}
{ \varphi}_n=
\left| {\begin{array}{*{20}{c}}
{ \varphi}_0&-1&0&0&0&\hdots&0&0\\
{ \varphi}_1&0&-1&0&0&\hdots&0&0\\
{ \varphi}_2&0&0&-1&0&\hdots&0&0\\
0&t&s&r&-1&\hdots&0&0\\
\vdots&\ddots&\ddots&\ddots&\ddots&\ddots&\vdots&\vdots\\
0&0&0&0&0&\ddots&r&-1\\
0&0&0&0&0&\ddots&s&r\\
\end{array}} \right|_{(n+1)\times(n+1)} 
\end{equation} 
\end{theorem}
\begin{proof}
We can finish the proof by using mathematical induction and equation \eqref{recurrencerelation}.
\end{proof}
Also, we give a numerical algorithm with respect to the Theorem \ref{det}:
\begin{table}[!ht]
\centering
\caption{A numerical algorithm for finding ${n{th}}$ term of the generalized Tribonacci hyperbolic spinor}\label{tab1}
\begin{tabular}{| l |}
  \hline  
  Numerical Algorithm 	\\ \hline	
{\bf{(1)}} Begin\\
{\bf{(2)}} Input $\varphi_0,\varphi_1$ and $\varphi_2$\\
{\bf{(3)}} Form $\varphi_n$ according to the equation \eqref{det1}  \\
{\bf{(4)}} Compute $\varphi_n$   \\
{\bf{(5)}} Output  $ \varphi_n\equiv\begin{bmatrix}
        V_{n}+V_{n+3}j\\
        -V_{n+1}+V_{n+2}j
        \end{bmatrix}$	 \\
        {\bf{(6)}} Finish 
\\ \hline   
\end{tabular}
\end{table}

Thanks to the Cereceda \cite{cerecedadet} and Cerda-Morales \cite{moralescomplex}, we establish the other method in Theorem \ref{det2thrm}, to calculate the $n{th}$ generalized Tribonacci hyperbolic spinor.
\begin{theorem}\label{det2thrm} For all $ n \in \mathbb{N}$, the following equation holds:
\begin{equation}\label{det11}
\varphi_n=
\left| {\begin{array}{*{20}{c}}
{ \varphi}_0&1&0&0&0&0&\hdots&0&0\\
r{\varphi}_0-{ \varphi}_1&r&\cfrac{1}{{ \varphi}_0}&0&0&0&\hdots&0&0\\
0&r{ \varphi}_1-{ \varphi}_2&r&t&0&0&\hdots&0&0\\
0&{ \varphi}_0&-\cfrac{s}{t}&r&t&0&\hdots&0&0\\
0&0&\cfrac{1}{t}&-\cfrac{s}{t}&r&t&\hdots&0&0\\
\vdots&\ddots&\ddots&\ddots&\ddots&\ddots&\ddots&\ddots&\vdots\\
0&0&0&0&0&0&\hdots&r&t&\\
0&0&0&0&0&0&\hdots&-\cfrac{s}{t}&r\\
\end{array}} \right|_{(n+1)\times(n+1)}
\end{equation}
\normalsize
where denominators are not zero.
\end{theorem}
\begin{proof}
We skip this proof for the sake of brevity.
\end{proof}
\begin{table}[!ht]
\centering
\caption{A numerical algorithm for finding ${n{th}}$ term of the generalized Tribonacci hyperbolic spinor}\label{tab2}
\begin{tabular}{| l |}
  \hline 
  {Numerical Algorithm} 	\\ \hline
{\bf{(1)}} Begin\\
{\bf{(2)}} Input $\varphi_0,\varphi_1$ and $\varphi_2$\\
{\bf{(3)}} Form $\varphi_n$ according to the equation \eqref{det11}  \\
{\bf{(4)}} Compute $\varphi_n$   \\
{\bf{(5)}} Output $ \varphi_n\equiv\begin{bmatrix}
        V_{n}+V_{n+3}j\\
       - V_{n+1}+V_{n+2}j
        \end{bmatrix}$	 \\
        {\bf{(6)}} Finish 
\\ \hline   
\end{tabular}
\end{table}

\newpage
Let us examine a special case of generalized Tribonacci hyperbolic spinors. We will examine the Jacobsthal-Padovan spinor sequence in the following Corollary \ref{cor1}:

\begin{corollary}\label{cor1}
Let us consider the $n^{th}$ Jacobsthal-Padovan hyperbolic spinor $\breve\upsilon_n$ with the initial values
\begin{equation*}
{\breve\upsilon}_0= \begin{bmatrix}
        1+3j\\
        -1+j
        \end{bmatrix},\quad
{\breve\upsilon}_1= \begin{bmatrix}
        1+3j\\
        -1+3j
        \end{bmatrix},\quad
{\breve\upsilon}_2= \begin{bmatrix}
        1+5j\\
        -3+3j
        \end{bmatrix}.
    \end{equation*}
     Then the following properties are satisfied:
\begin{itemize}
\item[\ding{93}] The recurrence relation of Jacobsthal-Padovan hyperbolic spinors is:
\begin{equation*}
     {{\breve\upsilon}}_n={{\breve\upsilon}}_{n-2}+2{{\breve\upsilon}}_{n-3}, \quad n\ge 3.
\end{equation*}
\item[\ding{93}] The generating function of Jacobsthal-Padovan hyperbolic spinors is:
    \begin{equation*}
    \sum\limits_{n = 0}^\infty {{\breve\upsilon}}_n {x^n}=\frac{{{{\breve\upsilon_0}} + {\breve\upsilon_1}x + ({{\breve\upsilon_2}}- \breve\upsilon_0){x^2}}}{{1-{x^2}-{x^3}}}=\cfrac{1}{{1-x^2-2x^3}}\begin{bmatrix}
            1+x+\left(3+3x+2x^2\right) j\\
            -1-x-2x^2+\left(1+3x\right)j
        \end{bmatrix}\raisepunct{.}
    \end{equation*}
  \item[\ding{93}] The Binet formula of Jacobsthal-Padovan hyperbolic spinors is:
    \begin{equation*}
    \begin{split}
        {\breve\upsilon}_n =& \frac{(\sigma_1+1)\sigma_1^{n+1}\zeta_1}{(\sigma_1-\sigma_2)(\sigma_1-\sigma_3)}+\frac{(\sigma_2+1)\sigma_2^{n+1}\zeta_2}{(\sigma_2-\sigma_1)(\sigma_2-\sigma_3)}+\frac{(\sigma_3+1)\sigma_3^{n+1}\zeta_3}{(\sigma_3-\sigma_1)(\sigma_3-\sigma_2)}\raisepunct{,}
            \end{split}
        \end{equation*}
        where (see $\sigma_1\sigma_2,\sigma_3$ in \cite{Soykangeneralizedjp})
        \begin{equation*}
           \left\{ \begin{split}
     \sigma_1&=  \left(1+\cfrac{\sqrt{78}}{9}  \right)^{1/3}+\left(1-\cfrac{\sqrt{78}}{9}  \right)^{1/3},\\
        \sigma_2&=\cfrac{-1+\sqrt{3}i}{2} \left(1+\cfrac{\sqrt{78}}{9}  \right)^{1/3}+\left( \cfrac{-1+\sqrt{3}i}{2} \right)^2\left(1-\cfrac{\sqrt{78}}{9}  \right)^{1/3},\\
         \varrho_3&=\left( \cfrac{-1+\sqrt{3}i}{2} \right)^2 \left(1+\cfrac{\sqrt{78}}{9}  \right)^{1/3}+\cfrac{-1+\sqrt{3}i}{2} \left(1-\cfrac{\sqrt{78}}{9}  \right)^{1/3}.
        \end{split}\right.
        \end{equation*}
      
         \item[\ding{93}] The exponential generating function of Jacobthal-Padovan hyperbolic spinors is:
\begin{equation*}
\begin{split}
     \sum\limits_{n = 0}^\infty {{\breve\upsilon}}_n \cfrac{y^n}{n!}=& \frac{(\sigma_1+1)\sigma_1\zeta_1 e^{\sigma_1y}}{(\sigma_1-\sigma_2)(\sigma_1-\sigma_3)}+\frac{(\sigma_2+1)\sigma_2\zeta_2e^{\sigma_2y}}{(\sigma_2-\sigma_1)(\sigma_2-\sigma_3)}+\frac{(\sigma_3+1)\sigma_2\zeta_3e^{\sigma_3y}}{(\sigma_3-\sigma_1)(\sigma_3-\sigma_2)}\raisepunct{.}
   \end{split}
\end{equation*}
\item[\ding{93}] The Poisson generating function of Jacobsthal-Padovan hyperbolic spinors is:
\begin{equation*}
\begin{split}
     e^{-y}\sum\limits_{n = 0}^\infty {\breve\upsilon}_n \cfrac{y^n}{n!}=&\frac{(\sigma_1+1)\sigma_1\zeta_1 e^{\sigma_1y}}{e^y(\sigma_1-\sigma_2)(\sigma_1-\sigma_3)}+\frac{(\sigma_2+1)\sigma_2\zeta_2e^{\sigma_2y}}{e^y(\sigma_2-\sigma_1)(\sigma_2-\sigma_3)}\\&+\frac{(\sigma_3+1)\sigma_2\zeta_3e^{\sigma_3y}}{e^y(\sigma_3-\sigma_1)(\sigma_3-\sigma_2)}\raisepunct{.}
   \end{split}
\end{equation*}
          \item[\ding{93}] For all $m \in \mathbb{N}$, the summation formulas of Jacobsthal-Padovan hyperbolic spinors are given:
         \begin{itemize}
            \item [\ding{99}]
             $\sum\limits_{n = 0}^m {\breve\upsilon}_{n} = \cfrac{1}{2}\left(\breve\upsilon_{m + 3} + \breve\upsilon_{m + 2} - \breve\upsilon_2 - \breve\upsilon_1\right), $
             \item [\ding{99}]
             ${\sum\limits_{n = 0}^m {{\breve\upsilon}}_{2n}} =\cfrac{1}{2}\left( {\breve\upsilon_{2m + 1}} + 2{\breve\upsilon_{2m}} - {\breve\upsilon_1}\right),$
             \item [\ding{99}]
             ${\sum\limits_{n = 0}^m {{\breve\upsilon}}_{2n+1}} =\cfrac{1}{2}\left({\breve\upsilon_{2m + 2}} + 2{\breve\upsilon_{2m+1}} - {\breve\upsilon_2}\right).$
         \end{itemize}
         \vspace{1mm}
\item[\ding{93}] The following property of Jacobtshal-Padovan hyperbolic spinors is obtained:
\begin{equation*}
  \left( {\begin{array}{*{20}{c}}
{ \breve\upsilon}_{n+2}\\
{\breve\upsilon}_{n+1}\\
{ \breve\upsilon}_{n}
\end{array}} \right)= \left( {\begin{array}{*{20}{c}}
0&1&2\\
1&0&0\\
0&1&0
\end{array}} \right)^n  {\left( {\begin{array}{*{20}{c}}
{ \breve\upsilon}_{2}\\
{ \breve\upsilon}_{1}\\
{ \breve\upsilon}_{0}
\end{array}} \right)}.
\end{equation*}
\item[\ding{93}] The following equations of Jacobsthal-Padovan hyperbolic spinors hold:
\begin{equation*}
{\breve\upsilon}_n=
\left| {\begin{array}{*{20}{c}}
{ \breve\upsilon}_0&-1&0&0&0&\hdots&0&0\\
{ \breve\upsilon}_1&0&-1&0&0&\hdots&0&0\\
{ \breve\upsilon}_2&0&0&-1&0&\hdots&0&0\\
0&2&1&0&-1&\hdots&0&0\\
\vdots&\ddots&\ddots&\ddots&\ddots&\ddots&\vdots&\vdots\\
0&0&0&0&0&\ddots&0&-1\\
0&0&0&0&0&\ddots&1&0\\
\end{array}}\right|_{(n+1)\times(n+1)}.
\end{equation*}
\begin{equation*}
{  \breve\upsilon}_{n}=
\left| {\begin{array}{*{20}{c}}
{  \breve\upsilon}_0&1&0&0&\hdots&0&0\\
-{ \breve\upsilon}_1&0&\cfrac{1}{{  \breve\upsilon}_0}&0&\hdots&0&0\\
0&-{  \breve\upsilon}_2&0&2&\hdots&0&0\\
0&{  \breve\upsilon}_0&-\cfrac{1}{2}&0&\hdots&0&0\\
0&0&\cfrac{1}{2}&-\cfrac{1}{2}&\hdots&0&0\\
\vdots&\ddots&\ddots&\ddots&\ddots&\ddots&\vdots\\
0&0&0&0&\hdots&0&1&\\
0&0&0&0&\hdots&-\cfrac{1}{2}&0\\
\end{array}} \right|_{(n+1)\times(n+1)}
\end{equation*}
\normalsize
where denominators do not vanish.
\end{itemize}
\end{corollary}

\section{A New Further Research: Generalized Tribonacci Polynomial Hyperbolic Spinors}\label{sec4}

In this section, we give a short introduction to some new arguments for where to go from here. For future works, let us make a brief introduction associated with the concept: generalized Tribonacci polynomial hyperbolic spinors, which include the generalized Tribonacci hyperbolic spinors.

Let ${V}_n(x)$ represents the $nth$ generalized Tribonacci polynomial. See for detailed information with respect to the generalized Tribonacci polynomials (see \cite{polynomial}).
\begin{definition} 
For $n\ge 0$, the $nth$ generalized Tribonacci polynomial hyperbolic spinor is determined as follows:
\begin{equation}
 \widetilde \psi_n(x)  =\begin{bmatrix}
        V_{n}(x)+V_{n+3}(x)j\\
        -V_{n+1}(x)+V_{n+2}(x)j
        \end{bmatrix}
\end{equation}
Since this transformation is linear and one-to-one but not onto, a new type sequence, which is a linear recurrence sequence, is constructed.
\end{definition}

\begin{theorem}[\textbf{Recurrence Relation}]
For all $n\ge0$, the following recurrence relations are written for the generalized Tribonacci hyperbolic spinor sequence:
\begin{equation}
   \widetilde \psi_{n}(x)=r(x)\widetilde \psi_{n-1}(x)+s(x)\widetilde \psi_{n-2}(x)+t(x)\widetilde \psi_{n-3}(x).
\end{equation}
\end{theorem}

The other equations given in the previous section can be obtained for generalized Tribonacci polynomial hyperbolic spinors. We intend to examine this special sequence in a future study.

\section{Conclusion}\label{conclusion}
In this study, we examined generalized Tribonacci hyperbolic spinors. Also, we constructed some algebraic properties and equalities concerning the conjugations. Then, we obtained some equations such as recurrence relation, Binet formula, generating function, exponential generating function, Poisson generating function, summation formulas, and matrix formulas. Further, we constructed the determinant equalities for computing the terms of this general sequence. Moreover, we established some numerical algorithms with respect to this special number system. Furthermore, we gave a short introduction with respect to the generalized Tribonacci polynomial hyperbolic spinors.

%%%%%%%%%%%%%%%%%%%%%%%%%%%%%%%%%%%%%%%%%%%%%%%%%%%%%%%%%%%%%%%%%%%%%%%%%%%%%%%%%%%

\end{document}